\documentclass{amsart}
\usepackage{amsmath, amssymb}
\usepackage{graphicx}
\newtheorem{Theorem}{Theorem}[section]

\newtheorem{Proposition}[Theorem]{Proposition}
\theoremstyle{definition}
\newtheorem{Definition}[Theorem]{Definition}
\theoremstyle{remark}
\newtheorem{rem}[Theorem]{Remark}
\numberwithin{equation}{section}

\newcommand{\R}{\mathbb R}
\newcommand{\N}{\mathbb N}
\newcommand{\Z}{\mathbb Z}

\begin{document}

\title[PC on groups: mathematical structures]{On Mathematical structures on pairwise comparisons matrices with coefficients in a group arising from  quantum gravity}
\author{Jean-Pierre Magnot}

\address{LAREMA, universit\'e d'Angers \\ Boulevard Lavoisier\\ F-49045 Angers cedex 01 \\ and    Lyc\'ee Jeanne d'Arc \\ 30 avenue de Grande Bretagne\\
F-63000 Clermont-Ferrand \\ http://orcid.org/0000-0002-3959-3443}
\email{jean-pierr.magnot@ac-clermont.fr}

\begin{abstract}
	We describe the mathematical properties of pairwise comparisons matrices with coefficients in an arbitrary group. We provide a vocabulary adapted for the description of main algebraic properties of inconsistency maps, describe an example where the use of a non abelian group is necessary. Algebraic, topological, geometric and probabilistic aspects are considered.
\end{abstract}

\maketitle

\section*{Introduction}
Pairwise comparisons are among the classical ways of decison making and information checking. The main idea is simple: assign a score to the comparison of a pair $(s,s')$ of two states which have to be compared, and we say that the ``scores''
are consistent if, for three states $(s,s',s'')$, the comparison of $s$ and $s''$
can be deduced from the comparisons of $s$ and $s',$ and of $s'$ and $s''.$ If not, the comparisons are called inconsistent. These aspects are precised in section \ref{2.1.}, and gives so many applications that it is impossible to cite them all. We mention two of them \cite{CD'AS2010,CW1985,LLA2012} which are applications of deep interest.  

This is mostly why some authors have developed ways to quantify inconsistency \cite{K1993,S1977}, see e.g. \cite{BR2008}, and there exists actually tentatives of axiomatizations of the so-called ``inconsistency indicators'' \cite{BF2015,B2016,KS2015}. One can also try to deal with partial orders, after \cite{J2009,JK1998,JZ2011,Z2010}, motivated by the obvious lack of informations when one tries to express a complex situation only by a , or by the necessary hierarchization of constraints in a difficult problem.

This leads us to the main motivation of this work. Dealing with partial orders can turn out to be very complex (see e.g.  \cite{Z2010}), where as (non abelian) groups furnish a minimal setting where composition and inversion are well-defined, with all the necessary properties for comparisons of more complex datas. Such an approach is already used in \cite{Fa2015,FV2012,Il2000,Yo99}, linked with so-called gauge theories in quantum physics \cite{Ma2018-1,Yo99}. Other applied approaches can be found in \cite{Mac95,R2015,STCDF17,TPCFT17} which shows the great vitality of this topic for applications in various fields. The aim of this paper is to draw-back the structures highlighted by physics and to adapt them in the context of pairwise comparisons matrices. As a by-product, the choice of the largest possible setting is natural, since we intend here to describe new perspectives of methods, even if most of pairwise comparisons matrices have coefficients in $\R_+^*.$ This is also the occasion to analyze whether classical objects in pairwise comparisons matrices arise from higher mathematical or physical considerations, or if they are only valid in the $\R_+^*-$setting. 
Let us now describe the contents of this paper.

Section 1 reviews selected topics on $\R_+^*-$pairwise comparisons matrices which will give rise to natural generalizations linked with a quantum physics picture. 

Section 2 describes pairwise comparisons matrices with coefficients in a group $G,$ indexed by any (finite or infinite) set of states $I.$ We define also what is an inconsistency map. The terminology of inconsistency indicator is reserved to inconsistency maps with additionnal properties along the lines of \cite{KMal2017}, and the settings developed will be justified by other parts of this work.

Section 3 deals with algebraic properties of consistencywhich extends as naturally as possible the classical setting recovered setting $G=\R_+^*$. We highlight adjoint action of the group $G^I$ , called gauge group by analogy with differential geometric settings \cite{KM2016a}, on the set of pairwise comparisons matrices. The key property states that consistency is characterized by an orbit of this group. The side-properties are then described, which motivates the vocabulary for properties of inconsistency maps, and leads to the terminology of inconsistency indicator. As a concluding property, we show that Koczkodaj's inconistency map is an adjoint-invariant inconsistency indicator. We are aware that a similar study, based on a deep understanding of the properties of Saaty's inconsistency map \cite{S1977}, is actually investigated by others. 
Section 4 deals with generalization on graphs. 
This happens when two states cannot be compared by direct comparisons, but only by comparisons with intermediate states. This leads to ``holes'' in the pairwise comparisons matrices, assigned to the coefficient ``0'', and the notion of holonomy enables us to extend Koczkodaj's inconsistency maps to a $\R[[X]]-$valued inconsistency map, which is adjoint-invariant.  

       In section 5, we first answer to \cite{KSW2016} where an approach by so-called $G-$distances is prposed. We feel the need to show that this approach does not seem natural for us an from the viewpoint of this work. More natural is the topological notion of filter. We show how inconsistency maps naturally give rise to a filter of so-called vincinities of the set $CPC_I(G).$ We cannot call it "filter of neighborhoods" because the notion of neighborhood is a very precise notion in topology, linked with the notion of continuity of maps. Here, the inconsistency indicators are not a priori continuous, this is why we feel the need to highlight the distinction by the chosen vocabulary.
       
      In section 6, motivated by Yamabe theorem that we recall briefly, we analyze sme geometric aspects of pairwise comparisons. The correspondence with gauge theories then arises clearly.
      The  objectives of this section are as follows.
      We consider the conditions of consistency and inconsistency in the geometric setting of a finite or infinite dimensional simplex (section \ref{3}), which can be understood as higher dimensional triangles (which are 2-simplexes) and tetrahedra (which are 3-simplexes), see e.g. \cite{Wh, Friedman}; each $s_i$ corresponds to  a 0-vertex; each 1-vertex gives an edge, and a triad (where inconsistency can be measured) is a 2-face { ( or $2-$simplex).}
      According to this setting, a geometric picture which is very similar to inconsistency is the 
      holonomy of a connection (see, for example, \cite{KMS} for holonomy in finite dimensions, and  \cite{Ma2015} for the infinite dimensional case). We show that a PC matrix $A$ can always be expressed as a holonomy matrix if the group $G$ is exponential, and when $G = (\mathbb{R}_+^*)^J,$  consistent PC matrices are holonomies of a flat connection which can be constructed. 
      
      We finish with probabilistic approaches and examples.
      Examples under consideration may be qualified as toy examples, given to give accessible situations which mathematical intuition can be compared with. Necessity of strongest examples seem unnecessary in view of the literature given in bibliography where simplicial or lattice
       gauge theories (i.e. pairwise comparisons) arise. Concerning probabilistic aspects, we restrict ourselves to interpret, in terms of pairwise comparisons matrices, two probabilistic approaches of gauge theories. The first approach presented relies cylindrical approximations of the Weiner measure, while the second approach intends to explain in a way as simple as possible a way to understand inconsistency indicators in a way parallel to Lagrangian theories and action functionals. These very technical aspects of mathematics used in physics are here simplified and adapted in the spirit of the whole paper.  
\section{Pairwise comparisons matrices with coefficients in $\mathbb{R}_+$} \label{2.1.}

It is easy to explain the inconsistency in pairwise comparisons when we consider cycles of three comparisons, called triad and represented here as $(x,y,z)$,
which do not have the ``morphism of groupoid'' property such as $$x.z \neq y$$ 
Evidently, the inconsistency in a triad $(x,y,z)$ is somehow (not linearly) proportional to $y-xz.$
In the linear space, the inconsistency is measured by the ``approximate flatness'' of the triangle. The triad is consistent if the triangle is flat.
For example, $(1,2,1)$ and $(10,101,10)$ have the difference $y-xz =1$ but the inconsistency in the first triad is unacceptable.It is acceptable in the second triad. In order to measure inconsistency, one usually considers coefficients $a_{i,j}$ with values in an abelian group $G,$ with al least 3 indexes $i,j,k.$  
\noindent The use of ``inconsistency'' has a meaning of a measure of inconsistency in this study; not the concept itself. 
The approach to inconsistency (originated  in \cite{K1993} and generalized in \cite{DK1994}) can be reduced to a simple observation:
\begin{itemize}
	\item 
	search all triads (which generate all 3 by 3 PC submatrices) and locate the worse triad with a so-called inconsistency indicator ($ii$),
	\item 
	$ii$ of the worse triad becomes $ii$ of the entire PC matrix.
\end{itemize}

\noindent Expressing it a bit more formally in terms of triads (the upper triangle of a PC submatrix $3 \times 3$), we have:
\begin{equation} \label{kii3}Kii(x,y,z) = 1-\min\left \{\frac{y}{xz},\frac{xz}{y}\right \}. \end{equation}

\noindent  According to \cite{KS2014}, it is equivalent to:
$$ii(x,y,z)=1- e^{-\left|\ln\left (\frac{y}{xz}\right )\right |}$$
\noindent The expression $|\ln(\frac{y}{xz})|$
is the distance of the triad $T$ from 0. When this distance increases, the $ii(x,y,z)$ also increases. It is important to notice here that this definition allows us to localize the inconsistency in the matrix PC and it is of a considerable importance for most applications.

Another possible definition of the inconsistency indicator can also be defined (following \cite{KS2014}) as:
\begin{equation} \label{kiin}{\rm Kii}_n(A)=1-\min_{1\le i<j\le n}  \min\left ({a_{ij}\over a_{i,i+1}a_{i+1,i+2}\ldots a_{j-1,j}},\,
{a_{i,i+1}a_{i+1,i+2}\ldots a_{j-1,j}\over a_{ij}} \right ) \end{equation}

\noindent since the matrix $A$ is consistent if and only if for any $1\le i<j\le n$ the following equation holds:
$$a_{ij}=a_{i,i+1}a_{i+1,i+2}\ldots a_{j-1,j}.$$

\noindent It is equivalent to:
\begin{equation} {\rm Kii}_n(A)=1-\max_{1\le i<j\le n}
\left (1 - e^{-\left |\ln \left ( {a_{ij}\over
		a_{i,i+1}a_{i+1,i+2}\ldots  a_{j-1,j}}  \right )\right |}\right )
\end{equation}

The first Koczkodaj's indicator $Kii_3$ allows us not only find the localization of the worst inconsistency but to reduce the inconsistency by a step-by-step process which is crucial for practical applications. The second Koczkodaj's indicator $Kii_n$ is useful when the global inconsistency indicator is needed for acceptance or rejection of the PC matrix. An abstract unification will be proposed at the end of section 4.

\section{Changing the comparisons structure to arbitrary groups}
In the previous section, the comparisons coefficients are $a_{i,j}$ are scaling coefficients. This means that,if the PC matrix $A$ is coherent, given a state $s_k,$ we can recover all the other states $s_j$ by something assimilated to scalar mutiplication: $$s_j = a_{j,k}s_k.$$
In other words, even if the states $s_j$ are driven by more complex rules, we reduce them to a ``score'' or an ``evaluation'' in $\R_+.$ This is useless to say that such an approach is highly reductive: even in video games, the virtual fighters have more than one characteristic: health, speed, strength, mental...
and the global design of these characteristics intends to reflect some ``complexity'' in the game (please note the `` ''). So that, the states $s_j$ have to belong to a more complex state space $S,$ and in order to have pairwise comparisons, a straightforard study shows that we define \cite{LawSch1997,McL} a semi-category $C_S,$
with set of objects $Ob(C_S)=S,$ and such that morphisms
$Mor(C_S)$ must satisfy the following properties:
\begin{itemize}
	\item there exists an identity morphism
	\item if $a\in Hom(C_S),$ then there exists $a^{-1} \in Hom(C_S)$
	\item any morphism acts on any state, which can be rephrased in the language of categories by: the semi-category is total.  
\end{itemize} 
Thus, gathering the necessary properties of $Mor(C_S),$ we get:
\begin{Proposition}
	$Hom(C_S)$ is a group.
\end{Proposition}   
This confirms the setting described in \cite[section 2]{KSW2016}.
Then we have that the minimal setting for generalizing pairwise comparisons is given by actions on $S$ by a group $G,$ which leads to the following setting. Let $I$ be a set of indexes and let $(k,+,.,|.|)$be a field with absolute value and $V_k$ a normed $k-$vector space.

\begin{Definition} \label{pcig} \cite{KM2016a,Ma2018-1}
	Let $(G,.)$ be a group. A \textbf{pairwise comparisons matrix} is a matrix $$A= (a_{i,j})_{(i,j)\in I^2} $$
	such that 
	
	\begin{enumerate}
		\item $\forall(i,j) \in I^2, a_{i,j}\in G.$

		\item  \label{inverse} $ \forall (i,j)\in I^2, \quad a_{j,i}=a_{i,j}^{-1}. $
		
		\item \label{diagonal} $a_{i,i} = 1_G.$ 
		
	\end{enumerate}
	
\end{Definition}
We note by $PC_I(G)$ the set of pairwise comparisons matrices indexed by $I$ and with coefficients in $G.$ When $G$ is not abelian, there are two notions of inconsistency:
\begin{itemize}
	\item $A$ is \textbf{covariantly consistent} if and only if $\forall (i,j,k)\in I^3, a_{i,k}=a_{i,j}a_{j,k}$
	\item $A$ is \textbf{contravariantly consistent} if and only if $\forall (i,j,k)\in I^3, \quad a_{i,k}= a_{j,k}a_{i,j}.$ 
\end{itemize}
Contravariant consistency appears in the geometric realization of $PC_I(G)$ via the holonomy of a connection on a simplex \cite{KM2016a}, but we give the following easy remark:
\begin{rem}
	Let $A=(a_{i,j})_{(i,j)\in I^2}$ be a contravariant PC-matrix. Then the $G-$matrix $B=(b_{i,j})_{(i,j)\in I^2}$ defined by $$\forall (i,j)\in I^2, \quad b_{i,j}= a_{i,j}^{-1}$$
	is a covariant PC matrix. 
\end{rem}
This shows that the two notions are dual, and we concentrate our efforts on covariant consistency in this section, that we call \textbf{consistency}.
  
\begin{Definition}
	Let $v_k$ be a vector space equipped with semi-norms.
A (non normalized, non covariant) inconsistency map is a map 
$$ii : PC_I(G) \rightarrow V_k$$
such that $ii(A)=0$ if $A$ is consistent. Moreover, we say that $ii$
is faithful if $ii(A) = 0$ implies that $A$ is consistent.  
\end{Definition}

We note by $CPC_I(G)$ the set of consistent PC-matrices.

\vskip 12pt
After that, since $V_k$ is a vector space equipped with a semi-norm, the semi-norm will give us the ``score" of inconsistency, as in the previous section.
One can assume for the sake of simplicity   that $V_k$ is a (normed) Euclidian space.

\section{Algebraic properties on $PC_I(G)$}
First, we give the following easy proposition:
\begin{Proposition}
Any morphism of group $a:G \rightarrow G'$ extends to a map $\bar{a}:PC_I(G) \rightarrow PC_I(G')$ by action on the coefficients, and:
\begin{itemize}
	\item If $A \in PC_I(G)$ is consistent, then $\bar{a}(A)\in PC_I(G')$is consistent.
	\item If $Ker (a) = \{e_G\},$ then $A \in CPC_I(G)$  if and only if $\bar{a}(A)\in CPC_I(G')$
	
\end{itemize}	
\end{Proposition}
 We call $G^I$ the \textbf{gauge group} of $G,$ following \cite{KM2016a}. Then we get the following actions: 
  
  \begin{itemize}
  	\item a left action $L : G^I \times PC_I(G) \rightarrow PC_I(G)$ defined, for $(g_i)_I \in G^I$ and $(a_{i,j})_{I^2}\in PC_I(G)$ by $$L_{(g_i)_I}\left((a_{i,j})_{I^2}\right) = (b_{i,j})_{I^2}$$
  	with $$b_{i,j} = \left\{\begin{array}{ccl} 1 & \hbox{ if } & i=j\\
  	g_i a_{i,j} & \hbox{ if } & i<j\\ 
  		 a_{i,j} g_{j}^{-1} & \hbox{ if } & i>j\\\end{array}\right.$$
   \item a right action $R :  PC_I(G) \times G^I\rightarrow PC_I(G)$ defined, for $(g_i)_I \in G^I$ and $(a_{i,j})_{I^2}\in PC_I(G)$ by $$R_{(g_i)_I}\left((a_{i,j})_{I^2}\right) = (b_{i,j})_{I^2}$$
   with $$b_{i,j} = \left\{\begin{array}{ccl} 1 & \hbox{ if } & i=j\\
    a_{i,j} g_j & \hbox{ if } & i<j\\ 
   g_i^{-1} a_{i,j}  & \hbox{ if } & i>j\\\end{array}\right.$$
   \item an adjoint action $$Ad_{(g_i)_I}=L_{(g_i)_I}\circ R_{(g_i)_I^{-1}} =  R_{(g_i)_I^{-1}}\circ L_{(g_i)_I}$$
   \item a coadjoint action $$\left((a_{i,j})_{I^2},(g_i)_I\right)\mapsto Ad_{(g_i)_I^{-1}}\left((a_{i,j})_{I^2}\right).$$
  \end{itemize}

 \begin{Theorem} \cite{KM2016a}\label{th1} When $I\subset \Z,$
 	\begin{equation*}
 	\exists (\lambda_i)_{i \in I}, \quad a_{i,j} = \lambda_i. \lambda_j^{-1} \Leftrightarrow A ~~is~consistent.
 	\end{equation*}  
 	
 \end{Theorem}
 We rephrase it the following way, extending it to any totally ordered set of indexes $I$: 
 \begin{Theorem} \label{th1'}
 	Consistent PC-matrices are the orbits of the PC-matrix $(1)_{I^2}$ with respect to the adjoint action.
 \end{Theorem}
  
  \begin{proof}
  	Let $A = (a_{i,j})_{I^2}$ be a consistent PC matrix. Let $i_0 \in I$ be a fixed index, and set $g_i = a_{i,i_0}. $
  	Since $A$ is consistent, \begin{eqnarray*}
  	a_{i,j} & = & a_{i,i_0}a_{i_0,j} \\
  	& = &   a_{i,i_0}a_{j,i_0}^{-1} \\
  	& = & g_i . 1. g_j^{-1}
  	\end{eqnarray*} 
  	\end{proof} 
  	
  	Let us give the following trivial proposition:
  	
  	\begin{Proposition}
  		$L$ and $R$ are effective actions.
  	\end{Proposition}
  	
  	One can wonder whether $L$ and $R$ are free or transitive. Let us consider the following "layered cake" example:
  	
  	Let $\lambda \in \R_+^*-\{1\}.$
  	Let us consider the matrix 
  	$$A=\left(\begin{array}{ccc} 1 & \lambda & \lambda^{-k} \\
  	\lambda^{-1} & 1 & \lambda \\ \lambda^k & \lambda^{-1} & 1 \end{array}\right).$$
  	Let us calculate the orbit of $A$ with respect to the left action (with the special case $G$ is abelian.) Let $g \in (\R_+^*)^3.$ $L_g(A) \in CPC_3(\R_+^*).$ $$L_g(A) =\left(\begin{array}{ccc} 1 & g_1\lambda & g_1\lambda^{-k} \\
  	g_1^{-1}\lambda^{-_k} & 1 & g_2\lambda \\ g_1^{-1}\lambda^k & g_2^{-1}\lambda^{-1} & 1 \end{array}\right)$$
  	In this example, $g_3$ is not acting, so that $L$ is not free.
  	Let us solve $$L_g(A) = \left(\begin{array}{ccc} 1 & 1 & 1 \\
  	1 & 1 & 1 \\ 1 & 1 & 1 \end{array}\right)$$  
  	we get the uncompatible equations:
  	$$\left\{\begin{array}{c} g_1 = \lambda^{-1} \\
  	g_2 = \lambda^{-1}\\
  	g_1 = \lambda^{k}\end{array}\right.$$
  	So that $A$ is not in the orbit of $\left(\begin{array}{ccc} 1 & 1 & 1 \\
  	1 & 1 & 1 \\ 1 & 1 & 1 \end{array}\right)$ for the left action, and hence the left action is not transitive. However, one could wonder whether the orbits of the left action intersect $CPC_I.$ With the same example, 
  	let us try to solve ``$L_g(A)$ is consistent'', we get 
  	$$ g_1g_2\lambda^2 = g_1\lambda^{-k}$$  which gives $$g_2 = \lambda^{-2-k}.$$
  	This gives a one parameter family of solutions 
  	$$ (g_1,g_2) = (t;\lambda^{-2-k}), \quad t \in \R_+^*.$$
  	Generalizing this, we give:
  	\begin{Theorem}
  		If $I=\N_3,$ any orbit for the left action intersects $CPC_3(G).$ If $card(I)>3,$ there exists orbits for the left action which do not intersect $CPC_I(G).$ 
  	\end{Theorem}
  	\begin{proof}
  		Let $A \in PC_3(G)$ and $g=(g_1,g_2,g_3)\in G^3$
  		$$L_g(A) =\left(\begin{array}{ccc} 1 & g_1a_{1,2} & g_1a_{1,3} \\
  			a_{2,1}g_1^{-1} & 1 & g_2 a_{2,3} \\  a_{3,1}g_1^{-1} & a_{3,2}g_2^{-1} & 1 \end{array}\right)$$
  		
  		We want to find $g$ in order to make $L_g(A)$ consistent. We get the following relation, among others:
  		$$\left\{\begin{array}{ccccl}  g_1a_{1,2}g_2 a_{2,3} & = & g_1a_{1,3} && \\
  		a_{2,1}a_{1,3}& = &a_{2,3}g_2^{-1}&&  \end{array}\right.$$
  		which gives
  		$$g_2= a_{2,1}a_{1,3} a_{3,2}$$
  		This condition gives the consistent PC-matrix:
  		$$L_g(A) =\left(\begin{array}{ccc} 1 & g_1a_{1,2} & g_1a_{1,3} \\
  		a_{2,1}g_1^{-1} & 1 & a_{2,1}a_{1,3} \\  a_{3,1}g_1^{-1} & a_{3,1}a_{1,2} & 1 \end{array}\right).$$
  		Let us now consider $A \in PC_4(G)$ and $g=(g_1,g_2,g_3,g_4)\in G^4.$ Then 
  		$$L_g(A) =\left(\begin{array}{cccc} 1 & g_1a_{1,2}&g_1a_{1,3} &g_1a_{1,4} \\
  		a_{2,1}g_1^{-1} & 1 & g_2 a_{2,3} & g_2 a_{2,4}\\ 
  		a_{3,1}g_1^{-1} & a_{3,2}g_2^{-1} & 1  & g_3 a_{3,4}\\\\
  		 a_{4,1}g_1^{-1} & a_{4,2}g_2^{-1} & a_{4,2}g_3^{-1}&  1 \end{array}\right).$$
  		 We then apply the procedure given for $3\times 3$ PC matrices on the diagonal $3\times3$ blocks. This gives, for $i \in \{2;3\}:$
  		 $$g_i = a_{i,i-1}a_{i-1,i+1} a_{i+1,i}$$
  		 and reporting this equality in the matrix, we get
  		 $$L_g(A) =\left(\begin{array}{cccc} 1 & g_1a_{1,2}&g_1a_{1,3} &g_1a_{1,4} \\
  		 a_{2,1}g_1^{-1} & 1 & a_{2,1}a_{1,3} & a_{2,1}a_{1,3}a_{3,2} a_{2,4}\\ 
  		 a_{3,1}g_1^{-1} & a_{3,1}a_{1,2} & 1  & a_{3,2}a_{2,4}\\\\
  		 a_{4,1}g_1^{-1} & a_{4,2}a_{2,3}a_{3,1}a_{1,2} & a_{4,2}a_{2,3}&  1 \end{array}\right).$$
  		 So that, consistency now depends on the first line and the first column, and we get the relations:
  		 \begin{eqnarray}
  		 g_1a_{1,2}a_{2,1}a_{1,3}a_{3,2} a_{2,4}& = & g_1a_{1,4}\\
  		 g_1a_{1,3}a_{3,2}a_{2,4} & = & g_1a_{1,4}
  		 \end{eqnarray}
  		 After simplifying $g_1,$ we gather the two lines give the same condition
  		 \begin{equation} a_{1,4} = a_{1,3}a_{3,2} a_{2,4}.\end{equation}
  		 This condition is not fulfilled, unless in very special cases. For an arbitrary $PC_I(G),$ with $card(I)>4,$ we extract a $4\times 4-$PC-matrix to get the same result. 
  		\end{proof}
  	Let us now turn to other properties inconsistency maps.
  	
  	\begin{Definition}
  		Let $ii$ be an inconsistency map. It is called:
  		\begin{itemize}
  			\item \textbf{normalized} if $\forall A \in PC_{I}(G), ||ii(A)||\leq 1.$
  			\item \textbf{Ad-invariant} if   $\forall A \in PC_{I}(G), \forall g \in G^I, ii\left(Ad_g(A)\right) =ii(A)$
  			\item \textbf{norm invariant} if $||ii(.)||$is Ad-invariant.
  		\end{itemize}
  	\end{Definition}

  	Let $F_I(G)$ be the quotient space for the Adjoint action of the gauge group on $PC_I(G).$ Next result is a classical factorization theorem:
  	\begin{Theorem}
  		An Ad-invariant inconsistency map $ii$ factors in an unique way through the maps $$ii = f \circ \pi $$
  		where 
  		
  		- $\pi : PC_I(G) \rightarrow F_I(G)$ is the quotient projection
  		
  		- $f \in V_k^{F_I(G)}.$
  	\end{Theorem}
  	
  	We give also the following easy proposition:
  	\begin{Proposition}
  		Morphisms of groups are acting by pull-back on inconsistecy indicators. Moreover, the pull-back of a normalized (resp. $Ad-$invariant) inconsistency map is a a normalized (resp. Ad-invariant) inconsistency map.
  	\end{Proposition}
  	
  	According to \cite{KS2015} and generalizing to any group $G$, we give now the following definition:
  	\begin{Definition}
  		An inconsistency indicator $ii$ on $PC_I(G)$ is a faithful, normalized inconsistency map with values in $\R_+$ such that there exists an inconsistency map $ii_3$ on $PC_3(G)$ that defines $ii$ by the following formula $$ ii(A) = \sup \left\{ ii_3(B) \ B \subset A ; \, \, \, B \in PC_3(G) \right\}.$$
  	\end{Definition}	
  	We remark here that since $ii$ is faithful, it is in particular (trivially) $Ad-$invariant on $CPC_I(G),$ but we do not require it to be $Ad-$invariant. Moreover, with such a definition, to show that $ii$ is $Ad-$invariant, it is sufficient to show that $ii_3$ is $Ad-$invariant.  However, we give the example 
  	driven by Koczkodaj's approach. This is already proved that $Kii_3$ generates an inconsistency indicator \cite{KS2015} and we complete this result by the following property:
  	\begin{Proposition} \label{Kad}
  		Let $n \geq 3.$ Koczkodaj's inconsistency maps $Kii_3$ and $Kii_n$ generate is $Ad-$invariant inconsistency maps on $PC_n(\R_+^*).$ 
  	\end{Proposition}
  	
  	\begin{proof}
  		This follows from straightforward computations of the type: 
  		$$ {\lambda_1 a_{1,3}\lambda_3^{-1}}
  		\left(\left(\lambda_1 a_{1,2}\lambda _2^{-1}\right)
  		\left(\lambda_2 a_{2,3}\lambda _3^{-1}\right)
  		\right)^{-1} = {a_{1,3}}a_{2,3}^{-1} a_{1,2}^{-1}.$$
  		\end{proof}

\section{Generalization: comparisons on a graph}

We consider in this section a family of states $(s_i)_I$ such that any $s_i$ cannot be a priori compared directly with any other $s_j.$ This leads us to consider a graph $\Gamma_I$ linking the elements which can be compared. For example, in the previous sections, $\Gamma_I$ was the $1-$skeleton of the simplex. For simplicity, we assume that $\Gamma_I$ is a connected graph, and that at most one vertex connects any two states $s_i$ and $s_j.$ We note this (oriented) vertex by $<s_i,s_j>,$ and the comparison coefficient by $a_{i,j}.$ By the way, we get a pairwise comparisons matrix $A$ indexed by $I$ with ``holes'' (with virtual $0-$coefficient) when a vertex does not exist, and for which $a_{j,i}^{-1} = a_{i,j}.$

\vskip 12pt
\noindent
\textbf{Example.}
	Let us consider the graph $\Gamma_5,$ with 5 states decribed figure \ref{fig:gamma5}.

	 A PC-matrix on $\Gamma_5$ is of the type:
	
	$$\left(\begin{array}{ccccc} 1 & a_{1,2} & a_{1,3} & a_{1,4} & a_{1,5} \\
 a_{1,2}^{-1} & 1 & 0 & 0 & 0 \\
 a_{1,3}^{-1} & 0 & 1 & 0 & 0 \\
 a_{1,4}^{-1} & 0 & 0 & 1 & a_{4,5} \\
 a_{1,5}^{-1} & 0 & 0 & a_{4,5}^{-1} & 1 \\
	\end{array}
	\right). $$

	\begin{figure}
		\centering
		\includegraphics[width=0.2\textwidth]{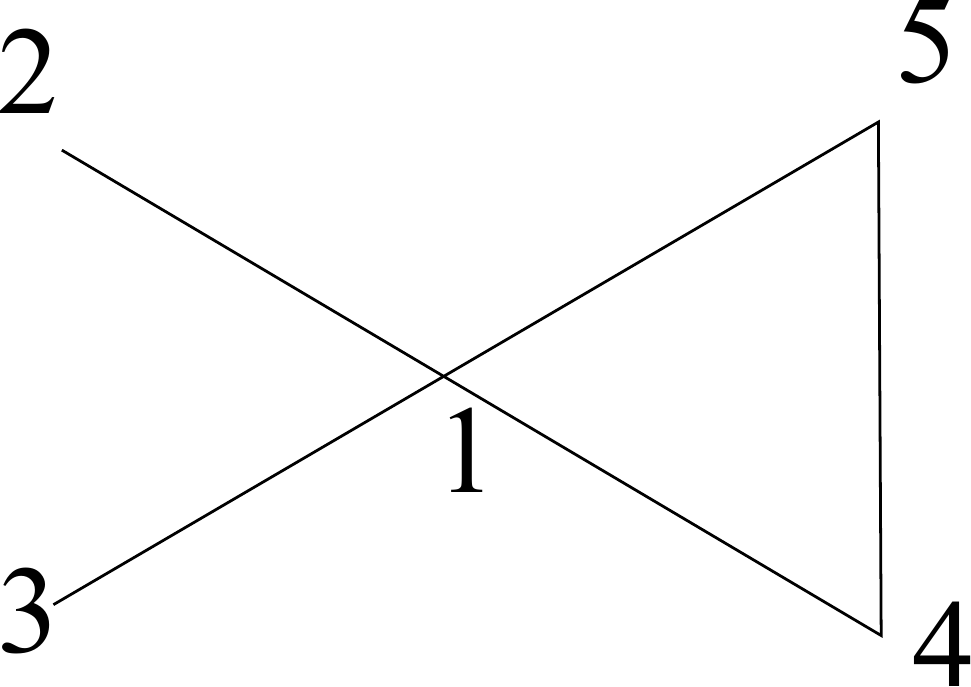}
		\caption[$\Gamma_5$]{The graph $\Gamma_5$}
		\label{fig:gamma5}
	\end{figure}

\subsection{Hierarchyless comparisons, ``hearsay'' evaluation and holonomy on a graph}
In this model, the comparison between two states $s_i$ and $s_j$ can be performed by any path between $s_i$ and $s_j$ of any length. One can think about the propagation of rumours, where validation of information is based on hearsay results. With this approach, the capacity of propagation of an evaluation is not controlled.
We note by $$<s_{i_1} ,...,s_{i_k}>= <s_{i_1} ,s_{i_2}>\vee...\vee <s_{i_{k-1}},s_{i_k}>$$ the composition of paths along vertices.
By analogy with the holonomy of a connection, we define:

\begin{Definition}
	Let $s=s_i$ and $s'=s_j$ be two states and let $$\mathcal{H}_{s,s'} = \left\{ a_{i,i_2}...a_{i_{k-1},i} | <s ,s_{i_2},...,s_{i_{k-1}},s'>  \hbox{ is a path from } s \hbox{ to } s'\right\}.$$
	We note by $\mathcal{H}_s$ the set $\mathcal{H}_{s,s}.$
	
\end{Definition}
 By the way, we get the following properties, usual for classical holonomy and with easy proof:

\begin{Proposition}
	\begin{enumerate}
		\item Let $s$ be a state, then $\mathcal{H}_s$ is a subgroup of $G.$ We call it holonomy group at $s.$
		\item Let $s$ and $s'$ be two states. Then $\mathcal{H}_s$ and $\mathcal{H}_{s'}$ are conjugate subgroups of $G.$
		\item  $$ \mathcal{H}_{s_i,s_j} = a_{i,j}.  \mathcal{H}_{s_j} = \mathcal{H}_{s_i} . a_{i,j}.$$
	\end{enumerate}
\end{Proposition}     

\vskip 12pt
\noindent
\textbf{Example.} With the graph $\Gamma_5$ of figure \ref{fig:gamma5}, $\mathcal{H}_1$ is the subgroup of $G$ generated by $a_{1,4}a_{4,5}a_{5,1}.$
\vskip 12pt
 
\begin{Definition}
The PC matrix $A$ on the graph $\Gamma_I$ is \textbf{consistent} if and only if there exists a state $s$ such that $\mathcal{H}_s = \{1\}.$
\end{Definition}

\subsection{Ranking the trustworthiness of indirect comparisons}
 The main problem with hierarchiless comparisons of two states $s $ and $s'$ is that paths of any length 
give comparison coefficients which cannot be distinguished. An indirect comparison, given by a path with 3 vertices, has the same status as a comparison involving a path with 100 vertices. This is why we need to introduce a grading on $H_{s,s'}$ called order. This terminology will be justified by the propositions thereafter. 

\begin{Definition}
	Let $s$ and $s'$ be two states.
	\begin{itemize}
		\item Let $\gamma$ be a path on $\Gamma_I$ from $s$ to $s'$. The \textbf{length} of $\gamma,$ noted by $l(\gamma),$ is the number of vertices of $\gamma,$ and by $H(\gamma)$ its holonomy.
		\item Let $h \in \mathcal{H}_{s,s'}.$ The \textbf{order} of $h$ is defined as $$ord(h) = \min \left\{ l(\gamma) \left| H(\gamma) = h \right. \right\}.$$ 
	\end{itemize}
\end{Definition}

As a trivial consequence of the triangular equality, and as a justification of the terminology, we have:
\begin{Proposition}
	Let $s,s'$ a,d $s''$ three states. Let $(h,h') \in \mathcal{H}_{s,s'} \times \mathcal{H}_{s',s''} .$ Then $$ord(hh') \leq ord(h) + ord(h').$$
\end{Proposition}

Left action, right action and adjoint action of $G^I$ extend straightway to PC-matrices on $\Gamma_I$ setting $$\forall g \in G, \quad g.0 = 0. g = 0.$$

Adapting the proof of Theorem \ref{th1'} we get:

\begin{Theorem}
	Let $A = (a_{i,j})_{(i,j)\in I^2}$ be a PC matrix on $\Gamma$ Then $A$ is consistent if and only if there exists $(\lambda_i) \in G^I$ such that
	$$ a_{i,j} = \lambda^{-1}_i \lambda_j$$
	when $a_{i,j} \neq 0.$
\end{Theorem}
  
Let $A$ be a PC matrix on $\Gamma_I.$ Inconsistency will be given here by the holonomy of a loop. Let us recall that a trivial holonomy of a loop $<s_{i_1},s_{i_2},..., s_{i_{k}},s_{i_1}>$ implies that $$a_{i_1,i_k}\left(a_{i_1,i_2}... a_{i_{k-1},i_k}\right)^{-1} = 1.$$  This relation has to be compared with formula (\ref{kiin}).  
The principle of ranking inconsistency with loop lengthgives the following:

\begin{Definition} \label{KiiN}
	Let $\mathcal{F} : G \rightarrow \R_+$ be a map such that $I(1)=0.$ Let $s$ be a basepoint on $\Gamma_I.$
	The \textbf{ranked Koczkodaj's inconsistency map} associated to $\mathcal{F}$ the map $$ Kii_{\N} = \sum_{n \in \N} a_n X^n$$
	where $$a_n = \sup \left\{ \mathcal{F}\left(H(\gamma)\right) \left| \gamma \hbox{ is a loop at } s \hbox{ and } \right. l(\gamma) = n \right\}.$$
\end{Definition}
 One can easily see that $a_n$ generalize $Kii_n,$ and $Kii_{\N}$ is a $\R[[X]]-$valued inconsistency map. Adapting Proposition \ref{Kad},
we get the following property:
\begin{Proposition}
	$Kii_n$ is an Ad-invariant inconsistency map if and only if $\mathcal{F}$ is Ad-invariant.
\end{Proposition} 
\section{Minimizing inconsistency and the necessity of topological structures on $G$.}
\subsection{Distance matrix} \label{4}

\noindent In this section, $G = \mathbb{R}_+^*.$
Setting $$k_{i,j} = | \log(a_{i,j}) |$$
we get another matrix, that we define as the \textit{distance matrix} $$K = (k_{i,j})_{(i,j)\in I^2}.$$
Notice that, if the coefficients of this matrix satisfy the triangle inequality
$\forall (i,j,l) \in I^3, k_{i,l} \leq k_{i,j} + k_{j,l},$
we get a curvature matrix for metric spaces \cite{Gro}. 
Due to the absolute value, we have the following:

\begin{Proposition}
	Let $K$ be a non zero distance matrix on $\Delta_n.$ Let $N$ be the number of non zero coefficients in $K. $ Then $N$ is even and there exist $2^{N/2}$ corresponding PC matrices.   
\end{Proposition}

\noindent
\textit{Outline of proof.} For each $k_{i,j}\neq 0,$ $\log (a_{i,j})$ has 2 possible signs.

\noindent Therefore, we have the following results: 

\begin{Proposition}
	Let $K$ be the distance matrix on $\Delta_n$ associated to a consistent PC matrix $A,$ which is assumed to be non zero. Let 
	$N'$ be the number of coefficients $k_{i,i+1}$ which are non zero. Then there exists $2^{N'}$ consistent PC matrices built with the coefficients $k_{i,i+1},$  but only $2$ consistent ones, $A$ and its transposition.   
\end{Proposition}

\begin{proof}
The first part of the proof follows the last proposition: the sign of $\ln( a_{i,i+1})$ gives the $2^{N'}$ consistent PC matrices which correspond to the coefficients $k_{i,i+1}$.
However, for any coefficient $k_{i,l},$ with $l>i,$ the formula $$a'_{i,k} = \prod_{j=i}^{l-1} a'_{j,j+1}$$ 
shows that there are two possible choices: $\ln a'_{0,1}= \ln a_{0,1}$ or $\ln a'_{0,1}= -\ln a_{0,1},$ which determines the sign of the other coefficients of the transposed matrix $A^T.$ \end{proof}


So that, PC matrices cannot be encoded as distance matrices in their own generalities when $G=\R_+^*.$ This suggests that the approach with $G-$distances suggested in \cite{KSW2016} does not generalize PC matrices but only describes a different, quite similar tool.   
\subsection{Vincinities, neighbourhoods of $CPC_I$ and inconsistency maps}
Let us remark that for $I=\N_3,$ $PC_3(G)$ splits by various ways into 
$$ CPC_3(G) \times G.$$
One map which realize this one-to-one correspondence is 
$$ \Phi_3 : \left(\begin{array}{ccc} 1 & a_{1,2} & a_{1,3} \\
a_{2,1} & 1 & a_{2,3} \\
a_{3,1} & a_{3,2} & 1\end{array}\right) \in PC_3(G) $$
$$ \mapsto \left( \left(\begin{array}{ccc} 1 & a_{1,2} & a_{1,2}a_{2,3} \\
a_{2,1} & 1 & a_{2,3} \\
a_{3,2}a_{2,1} & a_{3,2} & 1\end{array}\right), a_{1,2}a_{2,3}a_{3,1} \right) \in CPC_3(G) \times G.$$
By the way, $F_3(G)$ can be identified with $G.$
The same way, one can define:
$$ \Phi_n : (a_{i,j})_{(i,j) \in \N_n^2} \in PC_n(G) $$
$$ \mapsto \left( (b_{i,j})_{(i,j \in \N_n^2)}, a_{1,2}a_{2,3}a_{3,1}, a_{2,3}a_{3,4}b_{4,2},...,a_{1,2}a_{2,3}a_{3,4}a_{4,1},...  \right) \in CPC_n(G) \times G^{\frac{(n-1)(n-2)}{2}},$$
	where the $G-$coefficients $b_{i,j}$ are defined by :
	\begin{itemize}
		\item $b_{i,i} = 1_G,$
		\item $\forall i \in \N_{n-1},$ $b_{i,i+1}=a_{i,i+1}$ and $b_{i+1,1}=b^{-1}_{i,i+1}= a_{i+1,i},$
		\item $\forall i \in \N_{n-1},$ $\forall p \in \N_{n-1-i},$  
		$b_{i,i+1+p}= a_{i,i+1}a_{i+1,i+2}...a_{i+p,i+p+1}$ and $b_{i+1+p,i}=b_{i,i+1+p}^{-1},$
	\end{itemize}
which defines a consistent PC matrix. Notice that we recognize in the $G^{\frac{(n-1)(n-2)}{2}}-$component the coefficients present in the formulas for the generalization of Koczkodaj's inconsistency indicator given in Definition \ref{KiiN}. Moreover, changing the PC matrix  $(a_{i,j})_{(i,j) \in \N_n^2}$ by Adjoint action of an element $(g_i)_{i \in \N_n}$ of the gauge group $G^n$ into $(g_i.a_{i,j}.g_j^{-1})_{(i,j) \in \N_n^2}$ we get that:
	\begin{itemize}
	\item the matrix  $(b_{i,j})_{(i,j) \in \N_n^2}$ is also changed the same way into $(g_i.b_{i,j}.g_j^{-1})_{(i,j) \in \N_n^2}$
	\item the coefficients $a_{i,i+1}...a_{i,i+k}a_{i+k,i}$ in the $G^{\frac{(n-1)(n-2)}{2}}-$component are changed by classical Adjoint action of $G$ on itself, into  $g_i.a_{i,i+1}...a_{i,i+k}a_{i+k,i}g_i^{-1}.$
		\end{itemize}
We turn now to inconsistency maps. Here the notion of inconsistency indicator is not necessary since we principally look at what happens when the value of the inconsistency map is at a neighbourhood of 0. An inconsistency map $ii$ pulls back any topological filter in $ V_k$ into a topological filter in $PC_n(G).$ One of particular interest is the filter $\mathcal{F}(V_k)$ of neighbourhoods of $0$ in $V_k.$

  \begin{Definition}
  	The \textbf{fundamental filter} of an inconsistency map $ii$ is the filter $ii^*\left(\mathcal{F}(V_k)\right).$
  \end{Definition}
The following proposition is straightforward:
\begin{Proposition}
	An inconsistency map $ii$ on $PC_n(G)$ is faithful if and only if $$\bigcap_{A \in \mathcal{F}(V_k)} ii^{-1}(A)=CPC_n(G).$$
\end{Proposition}
Indeed the very wide variety of such filters around $CPC_n(G)$, even if inconsistency maps $ii$ can be, in a first approach, assumed faithful and $C^0-$maps, shows that there is at this step no way to decide how an inconsistency map can be better than another. However, the notion of Ad-invariance of inconsistency maps can furnish a preferred restricted class. Gathering the results given before, we get that $$F_n (G) \sim G^{\frac{(n-1)(n-2)}{2}}.$$
Remark that this identification is not $Ad-$invariant (unless $G$ is abelian), where as $F_n(G)$ is precisely the space of the orbits of the Adjoint action of the gauge group $G^n.$
So that,one can generate faithful $Ad-$invariant inconsistency maps from functions  $f:  F_n(G) \rightarrow V_k$ such that $f^{-1}(0)$ is the orbit $CPC_n(G).$ In this case, orbits are level lines of $ii = f \circ \pi.$  Koczkodaj's inconsistency indicator is such an example.
\section{Differential geometric methods on PC matrices}
We now recall a techincal but well-known result:

\begin{Theorem}[Yamabe's Theorem] \cite{Ya}
	Any locally compact topological subgroup of an analytic Lie group is an analytic Lie group.
\end{Theorem}
Hence, from the natural topological properties derived from the previous section, we can assume with almost no loss of generality that $G$ is a Lie group, at least for groups which have a presentation into matrix groups. 
\subsection{PC matrices read on a simplex} \label{3}



An exposition on holonomy is given in \cite{Lich,KMS,KM,Ma2015}. 
The geometry of simplexes is well addressed by \cite{Friedman, Wh}, and the notion of Lie group is described in \cite{Lich,KM}.
Examples of finite dimensional Lie groups are provided by (classical) groups of (invertible) matrices, where multiplication and inversion are smooth coefficientwise. Other examples can be provided using a very general framework of differentiable manifolds, but finite dimensional Lie groups can be realized as groups of matrices. 

Let $(G,.)$ be a Lie group with Lie algebra $(\mathfrak{g}, +, [.,.]).$
The expression ``Lie group'' is here understood in a very general sense. This can be a finite dimensional or an infinite dimensional group, or even a Fr\"olicher group with Lie algebra \cite{Ma2015}.   However, the beginner in the topic of Lie group is strongly advised to consider $G$ as a matrix group, and we furnish in the appendix a short intuitive introduction of Lie groups for beginners in order to help if needed, and can skip the end of this short paragraph. The only technical requirement for the sequel is the existence of an exponential map $$exp: C^\infty([0;1], \mathfrak{g}) \rightarrow C^\infty([0;1], G)$$
solving the logarithmic equation $g^{-1}.dg = v,$ where $ v \in C^\infty([0;1], \mathfrak{g})$ and $ g \in C^\infty([0;1], G).$ This ensures the existence of the holonomy of a connection \cite{Ma2013}. Such a property is always fulfilled for finite dimensional groups, but not for Fr\"olicher Lie groups. We get an example of Fr\"olicher Lie group with no exponential map considering $G=Diff_+(]0;1[),$ the group of increasing diffeomorphisms of the open unit interval \cite{Ma2018-2}.

On a trivial principal bundle $P = M \times G,$ the horizontal lift of a path $c \in C^\infty([0;1],M)$ from a starting point $p = (c(0),g_0) \in M \times G$ with respect to a connection $\theta$ is the path $\tilde c = (c,g) \in C^\infty([0,1],P)$ such that $$ g^{-1}.dg = \theta(dc).$$ If $c$ is a loop, we have $Hol_{g(0)}c = g(0)^{-1}. g(1).$ The holonomy of a loop depends on the base point $(\gamma(0),g(0))$ and is invariant under co-adjoint action. In this text, the coadjoint action is understood as an action of $G$ on the (total) space $TG = G \times \mathfrak{g}$  in the spirit of \cite{Mil}, which allows the same notation for the co-adjoint action of $G$ on itself or on its Lie algebra $\mathfrak{g}.$

\begin{figure}[t]
	\centering
	\includegraphics[width=0.9\textwidth]{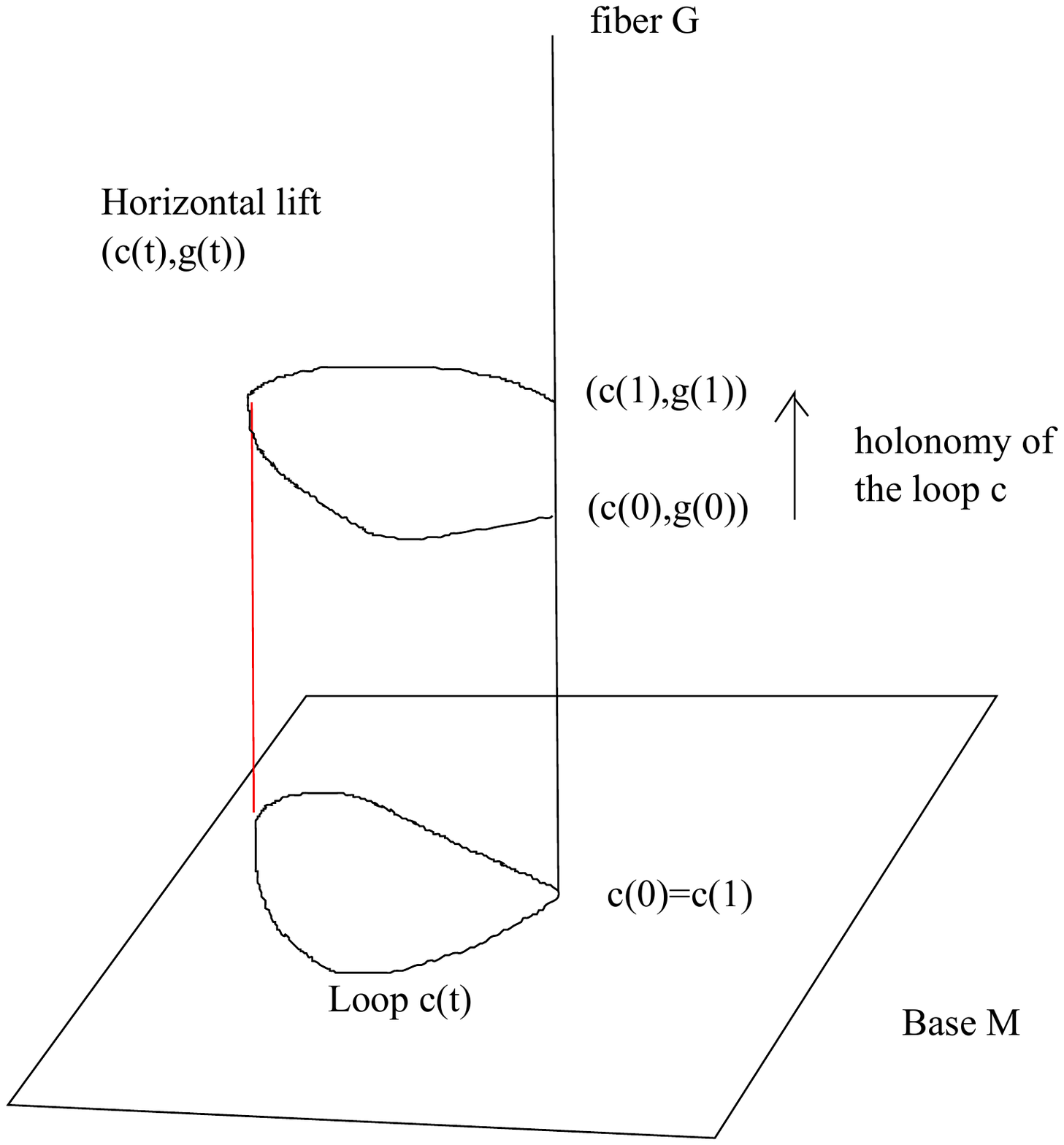}
	\caption[Holonomy]{holonomy of a loop c(t)}
	\label{fig:holonomy}
\end{figure}

\noindent Let $n \in \mathbb{N}^*$
and $$\Delta_n = \left\{ (x_0,\ldots, x_n) \in \mathbb{R}^{n+1} | \left(\sum_{i=0}^n x_i = 1\right) \wedge \left(\forall i \in \{0,\ldots.n\}, x_i \geq 0\right)\right\}$$
be an $n-$simplex.
This simplex can be generalized to the infinite dimension: 
$$\Delta_{\mathbb{N}} = \left\{ ( x_n)_{n \in \mathbb{N}} \in l^1(\mathbb{N}, \mathbb{R}_+^*) | \sum_{i=0}^\infty x_i = 1\right\}$$
\noindent	and 	$$\Delta_{\mathbb{Z}} = \left\{ ( x_n)_{n \in \mathbb{Z}} \in l^1(\mathbb{Z}, \mathbb{R}_+^*) | \sum_{i \in \mathbb{Z}} x_i = 1\right\},$$
where the summation over $\mathbb{Z}$ is done by integration with respect to the counting measure.
In the sequel, $\Delta$ will denote $\Delta_n,$ $\Delta_{\mathbb{N}}$ or $\Delta_{\mathbb{Z}}.$ Since $\Delta$ is smoothly contractible, any $G-$principal bundle over $\Delta$ is isomorphic to $\Delta \times G$
and a $G-$connection 1-form on $\Delta$ is a 1-form $\theta \in \Omega^1(\Delta ,\mathfrak{g}),$ which extends to a $G-$covariant 1-form in   $\Omega^1(\Delta ,\mathfrak{g}),$ with respect to the coadjoint action of $G$ on $\mathfrak{g}.$ We define a gauge $({g_i})_{i \in I}\in G^I$ with  $\widetilde{\gamma_i}(1) = (\gamma_i(1),{g_i})$ where $$\gamma_i = [s_0,s_1]*\ldots*[s_{i-1},s_i] \hbox{ if } i>0$$
\noindent and  $$\gamma_i = [s_0,s_{-1}]*\ldots*[s_{i+1},s_i] \hbox{ if } i<0.$$
We set $g_i = Hol_{(s_0,1_G)}\gamma_i.$
Let us recall that, for two paths $c$ and $c'$ such that $c*c'$ exists (i.e. $c(1)=c'(0)$), if $p=(c(0),e_G),$ $p'=(c'(0),e_G)$ and $h = Hol_{p}c,$ we have:

\begin{equation} 
\label{concat} 
Hol_p(c*c') = Hol_p(c) . (h^{-1}Hol_{p'}(c') h) = Hol_{p'}(c'). Hol_p(c).
\end{equation}
\\
\noindent Let \begin{equation} \label{aij} 	a_{i,j} = s_j.Hol_{(s_0;e_G)}\left( \gamma_i [s_i, s_j] \gamma_j^{-1} \right).s_i^{-1}. \end{equation}

In the light of these specifications, we set
\noindent $$ A = Mat(a_{i,j})$$
and the required notion of consistency is contravariant consistency.

\begin{Proposition}
	$A$ is a PC matrix.
\end{Proposition}

\begin{proof} This follows from holonomy in ``reverse orientation''. \end{proof}

Let $\gamma_{i,j,k} = \gamma_i * [s_i,s_j]*[s_j,s_k] * [s_k,s_i] *\gamma_i^{-1}$
be the loop based on $s_i$ along the border of the oriented 2-vertex $[s_i,s_j,s_k],$ where $*$ is the composition of paths.
By Equation (\ref{concat}), contravariant consistency seems to fit naturally with flatness of connections: 
\begin{eqnarray*} \forall i,j,k, & \quad & a_{i,k} =  a_{j,k}.a_{i,j} \\
	&\Leftrightarrow & a_{k,i} . a_{j,k} . a_{i,j} = a_{i,i} = 1_G \\
	& \Leftrightarrow & Hol (\gamma_{i,j,k}) = 1_G
\end{eqnarray*}


It is necessary to examine whether every PC matrix can be 
expressed as a matrix of holonomies of a fixed connection. 
For this, we need to assume that the group $G$ is exponential, 
which means that the exponential map $\mathfrak{g} \rightarrow G$ is onto. 

\begin{Proposition}
	If $G$ is exponential, the map \begin{eqnarray*}
		\Omega^1(\Delta, \mathfrak{g}) & \rightarrow & \{ \hbox{ PC matrices }\} \\
		\theta & \mapsto & \hbox{ the holonomy matrix }
	\end{eqnarray*}
	is onto.
\end{Proposition}

\noindent
\begin{proof}Let $A = (a_{i,j})_{(i,j)\in I^2}$ be a PC matrix. Let us build a connection 1-form $\theta \in \Omega^{1}(\Delta, \mathfrak{g})$ such as (\ref{aij}). For this, before constructing our connection, we fix the gauge $({g_i}_{i \in I}) \in G^I$ by $${g_i} = a_{0,1}\ldots a_{i-1,i}\hbox{ for } i>0$$
and by $${g_i} = a_{0,-1}...
a_{i+1,i}\hbox{ for } i<0.$$
Once the gauge is fixed, we begin by dealing with each 1-vertex and use gauge covariance to extend the 1-form in $\Omega^1(\Delta,\mathfrak{g})$ constructed to a ($G-$equivariant) connection 1-form on $\Delta \times G. $

Firstly, by fixing indexes $i<j,$ which holds in particular for $j = i+1,$ we choose $v_{i,j} \in \mathfrak{g}$ such that $\exp (v_{i,j}) = a_{i,j}.$ Needless to say, the condition $v_{j,i} = - v_{i,j}$ is consistent with $a_{i,j}^{-1}= a_{j,i}.$ The group $\{exp(tv_{i,j})|t \in \mathbb{R}\}$ is an abelian subgroup of $G.$ 
For this reason, formulas for holonomy on an abelian group can be used to specify a function $f_{i,j}: [0,1] \rightarrow \mathbb{R}_+ . v_{i,j}$, with support in $[1/3;2/3],$ and such that $\int_0^1 f_{i,j}(s)ds = v_i$ on the length-parametrized edge $[s_i,s_j].$ Finding such a function is possible, and extending the $G-$equivariant 1-form $f_{i,j}d_s$ on $[s_i,s_j]\times G$ to a $G-$equivariant 1-form $\theta_{i,j}$ on $\Delta \times G$ which is null off $V_{i,j}\times G$, where $V_{i,j}$ is a tubular neighborhood of radius $\epsilon>0$ of $supp(f_{i,j}),$ is also possible. 

Secondly, we repeat this procedure for each couple of indexes $(i,j)$ such that $i<j,$ and choose $\epsilon$ small enough in order to have non intersecting supports $supp(\theta_{i,j}),$ for example $\epsilon = 1/6.$ By setting $$\theta = \sum_{i<j}\theta_{i,j},$$
we get a connection $\theta$ whose holonomy matrix is given by $A$. \\ 
\end{proof}

\vskip 12pt
\noindent Let us provide a geometric criterion for consistency.
\vskip 6pt

\begin{Proposition}
	If the connection $\theta$ is flat, $A$ is a contravariant consistent PC matrix.
\end{Proposition}

\noindent
\begin{proof}
$\theta$ is flat if and only if its curvature is null. 
This implies that the Lie algebra of the holonomy group is null, and since each 2-vertex $[s_i,s_j,s_k]$ is contractible,  the holonomy group is trivial. \end{proof}
\subsection{The gauge group acting on the space of contravariant PC matrices}

The gauge group is defined before, and we give the following easy result which justifies the terminology:

\begin{Theorem}
	The coadjoint map
	$$\begin{array}{cccc}
	Ad: & G^I \times \{\hbox{contravariant PC matrices} \}& \rightarrow & \{\hbox{contravariant PC matrices} \} \\
	& (\tilde{g}_i)_{i \in I}\times (a_{j,k})_{(j,k)\in I^2} & \mapsto & (\tilde{g}_k^{-1}a_{j,k}\tilde{g}_j)_{(j,k)\in I^2}
	\end{array}$$
	transforms consistent PC matrices to consistent PC matrices. Moreover, for each $(\tilde{g}_i)_{i \in I} \in G^I,$ the map $Ad_{(\tilde{g}_i)_{i \in I}}$ is one-to-one and onto.
\end{Theorem} 

Let us now recall some basics on the gauge group on $\Delta \times G.$ This group is given by $C^\infty(\Delta,G)$ and acts on the space of connections $\Omega^1(\Delta,G)$ by the formula:
$$(g,\theta) \in C^\infty(\Delta,G)\times\Omega^1(\Delta,G) \mapsto g^{-1}dg + Ad_{g^{-1}}\theta.$$
Under these conditions, the holonomy $Hol_p(\gamma)$ of a loop $\gamma$ transforms  into $Ad_{g^{-1}}\left(Hol_p(\gamma)\right)$ at the same basepoint $p$ of the principal bundle $\Delta \times G.$ 

\begin{Theorem}
	If $G$ is a compact exponential finite dimensional Lie group, then for each $(\tilde{g}_i)_{i \in I} \in G^I,$ there exists $g \in  C^\infty(\Delta,G)$ such that $g(s_i)=g_i.$ Moreover, for any contravariant PC matrix $A$ which is the holonomy matrix of a connection $\theta,$ then $Ad_{(\tilde{g}_i)_{i \in I}}(A)$ is the holonomy matrix  of the connection $g^{-1}dg + Ad_{g^{-1}}\theta.$
\end{Theorem}

\noindent
\begin{proof} Since $G$ is exponential, setting $cut(G)$ to the cut-locus of $G$ with respect to the exponential map, we have that $G-cut(G)$ is star-shaped. So, there exists a continuous map $g: \Delta \rightarrow G$ which restricts to $\Delta - \partial \Delta \rightarrow G - cut(G)$ such that $g(s_i) = \tilde{g}_i.$ Since $G$ is compact (and hence finite-dimensional), this map can be chosen smooth \cite{Hir}. Once this map is constructed, the rest of the theorem follows from classical properties of the gauge group action that we have sketched before. \end{proof}
\subsection{When $G$ is a free abelian Lie group}
\noindent Assume that $G= \left(\mathbb{R^*}\right)^J$ where $J$ has any cardinality, finite or infinite. In this case,  $$Hol(<s_i,s_j>) = e^{\int_0^1 \theta(ds_{i,j})}$$
where $ds$ is the unit vector of the normalized length parametrization of $[s_i,s_j].$ Thus,
\begin{eqnarray} Hol(\gamma_{i,j,k}) = 1_G & \Leftrightarrow &  \int_0^1 \theta(ds_{i,j}) + \int_0^1 \theta(ds_{j,k}) + \int_0^1 \theta(ds_{k,i})=0
\end{eqnarray}

The connection $\theta$ is flat now and it reads as $ d \theta = 0$
which is equivalent to $\theta = df,$
where $f \in C^\infty(\Delta , \mathbb{R}^J)$ (because $H^1(\Delta, \mathbb{R})=0$). With this function $f$, setting $e^{f(s_i)}= \lambda_i,$ 
we recover the ``basic consistency condition'' of Theorem \ref{th1}.
In the spirit of  Whitney's simplicial approximation \cite{Wh},
we assume that $G = \mathbb{R}^* $ and $\mathfrak{g} = \mathbb{R}$ for simplicity, and our computations will extend to $\mathbb{R}^J$ componentwise.
Let us construct an affine function $f$. This function is uniquely determined by its values $f(s_i),$ for $i \in \{0,\ldots,n\}$ and we get the system: \\
\begin{equation} \label{eqf} \left\{\begin{array}{cccccc} f(s_0) & - f(s_1)&&&= &-\ln(a_{0,1}) \\
&f(s_1) & - f(s_2)&&= &-\ln(a_{1,2}) \\ 
&&&(...)&& \\
&&f(s_{n-1})&-f(s_n)&=& -\ln(a_{n-1,n}) \end{array}\right. \end{equation}
\medskip

\noindent which is a $n-$ system with $(n+1)$ variables. Since $\theta = df,$ we can normalize it, assuming e.g. $f(s_0) = 0,$ and the system gets a unique solution, and hence a unique affine function $f$ and an unique connection $\theta = d f.$
Now, setting $\lambda_i = e^{f(s_i)},$ we recover the construction given in the proof of theorem \ref{th1} for this particular choice of group $G.$

\subsection{Remarks on methods for minimizing inconcistency}
Let us fix the set $PC_n(G)$ and an inconsistency indicator $ii.$
We assume here that the chosen inconsistency indicator $ii$ is $C^0,$ $\R_+-$valued and of $C^k-$regularity off the set $CPC_n(G),$ with $k$ large enough. Passing from $G=\R_+^*$ to a general Lie group $G,$ most techniques used in the existing scientific literature on pairwise comparisons cannot be applied straightway because they are based on three properties of $\R_+^*$ which are still not valid on a Lie group $G,$ namely:

\begin{enumerate}
	\item $\R_+^*$ is contractible;
	\item $\forall x \in \R_+^*,$ $\exists ! y \in \R_+^*,$ $y^n=x$; 
	\item $\R_+^*$ is abelian.
\end{enumerate}

However, the inconsistency indicator $ii$ appears like a functional which has to be minimized. This kind of problem has been approached by many ways in optimization techniques, where enough regularity of the functional is needed. We leave these questions for future works, because the framework of application may impose to choose an optimization scheme or another. We also have to highlight that this problem of minimization for functional has a pending problem in the field of mathematical physics. In particular, when $ii$ is Koczkodaj's inconsistency indicator,  minimizing $$ii(A)= sup\{ii_3(B)\, | \, B\subset A\}$$ is similar to minimizing the holonomy along boundaries of 2-simplexes. This approach leads to consider Yang-Mills type functionals along the lines of \cite{Ma2018-1,Ma2018-3}, which will be studied extensively in a future work.

\section{Uncertainty and precision of measurements in pairwise comparisons}

We now turn to probabilistic aspects. For this we must assume, a priori, that there exists a \textbf{bi-invariant Radon measure } $\lambda$ on the topological group $G,$ which is also called \textbf{Haar measure}. For such a measure, for any integrable function $f:G \rightarrow \R,$ we have the bi-invariance properties of integration with respect to $\lambda:$ 
$$ \forall g \in G, \quad \int_G f(x)d\lambda(x) = \int_G f(g.x)d\lambda(x) = \int_G f(x.g)d\lambda(x).$$
For example, \begin{enumerate}
	\item for $n \in \N^*,$ (classical) Riemannian integration of functions $\R^n \rightarrow \R$ envolve the Haar measure of the abelian group $(\R^n,+)$ called Lebesgue measure.  
	\item the multiplicative group $(\R_+^*,.)$ has a Haar measure different from the Lebesgue measure on $(\R,+),$ because the group law is different. The exponential map $$ exp: \R \rightarrow \R_+^*$$ is an isomorphism of topological groups. Setting $\lambda$ and $\lambda'$ respectively the Lebsegue measure on $\R$and the Haar measure on $\R_+^*,$ if $g: \R_+^* \rightarrow \R$ is a $C^0-$function and $f = g \circ exp,$ we have $$\int_\R f d\lambda = \int _{\R_+^*} g d\lambda' = \int_0^{+\infty} \frac{g(x)}{x}dx .$$
	\item If $G$ is compact, the Haar measure has finite volume and can be normalized into a probability measure. 
\end{enumerate}
On one hand, considering only the Haar measure would mean that the measurements are totally random, without any consideration on the real situation under evaluation. On the other hand, exact evaluations belong to an idealist picture, and errors can occur from many different ways: measurements, expert unability among others. Moreover, error perception is limited to a fixed precision. These are the reasons why a measurement is better modelized by a probability distribution which "concentrates" around its expectation value which is chosen as th exact value of the measurement. Among the most known examples, Gaussian measures play a very important role. On $\R$, a Gaussian probability measure $\gamma$ such that $\mathbb{E}(\gamma)=0$ reads as  
\begin{equation} \label{gauss} \gamma(A) = \frac{1}{\sqrt{2\pi\sigma^2}}\int_A e^{-\frac{x^2}{2\sigma^2}} d\lambda.\end{equation}
When $\sigma \rightarrow 0^+,$ $\gamma$ converges, for the vague topology of probability measures, to the Dirac measure $\delta_0$ defined by $\delta_0(A) = 1$ if $0 \in A$, and $\delta_0(A) = 0$ if $0 \notin A.$ In this picture, $0$ reads as the exact measure, and $\mu(A )$ reflects with which probability the measurement obtained is in the set $A \subset \R.$ 
 
 We propose here two ways of generalization of this picture to pairwise comparisons matrices. In these two approaches, we assume that each coefficient $a_{i,j}$ is a random variable.
 \begin{itemize}
 	\item  On one hand, let us consider a measure $\gamma_{i,j}$ on $G$
 	which models, for any measurable set $U \subset G,$ the probability to get $a_{i,j} \in  U$ after "expert evaluation". Identifying $PC_n(G)$ with $G^{\frac{n(n-1)}{2}}$ we get a product measure $\otimes_{1\leq i < j \leq n}\gamma_{i,j}$ on $PC_n(G).$ Thus for any inconsistency map $$ii: PC_n(G) \rightarrow \R_+,$$ assuming that inconsistency of $A \in PC_n(G)$ is acceptable if $ii(A) < \epsilon$ for a fixed value $\epsilon >0,$ the probability for the acceptable inconsistency is given by $$\otimes_{1\leq i < j \leq n}\gamma_{i,j}\left(\left\{ A \in PC_n(G) \, | \, ii(A)\leq \epsilon \right\}\right).$$
 	\item On the other hand, let us assume that the chosen inconsistency map $ii$ reflects the global perception of the inconsistency of the system of an extra-observer, whose natural tendency will lead him to minimize inconsistency, ideally at the level $ii < \epsilon.$Then we can propose a Feyn-man-Kac type formula by introducing the (normalized) measure $\mu_{ii}$ of density $e^{-\frac{ii(a)}{\epsilon} }$ with respect to the Haar probability measure $\lambda$ on $G^{\frac{n(n-1)}{2}} \sim PC_n(G)$ given by the formula:
 	$$ \mu(A) = \frac{1}{Z}\int_A e^{-\frac{ii(x)}{\epsilon}}d\lambda(x)$$
 		with $$Z = \int_{PC_n(G)} e^{-\frac{ii(x)}{\epsilon}}d\lambda(x)!;$$
 	  
 \end{itemize}
  \section{Examples}
  \subsection{$GL_n$-comparisons}
  Let $S=\R^\infty$ be the inductive limit of the  family $\{\R^n ; n \in \N^*\}$ such that the inclusions $\R^n\subset \R^{n+1}$ is the canonical inclusion with respect to the first coordinates. Here, $S$ is an object of the category of vector spaces. With this setting, we get $G=GL_\infty(\R),$ which is the inductive limit of the family $\{GL_n(\R); n \in \N^*\}.$ If $I$ is a finite set of indexes (e.g. $I = \N_k$ for some $k \in \N^*$), $\sup_{i \in I} dim(s_i) < +\infty$ and setting $n = \sup_{i \in I} dim(s_i),$ we work with the restricted setting   $S = \R^n $ and $G = GL_n{(\R)}.$
  
  \begin{rem}
  	Even if there exists an index $i \in I$ such that $n_i < n,$ we have to consider the inclusion $\R^{n_i}\subset \R^n$ because there is no linear isomorphism from $\R^{n_i}$ to $\R^n$ (by the theorem of dimension). 
  \end{rem}
  
  With this construction, we get a first family of inconsistency maps. The determinant map $$det: GL_n{(\R)} \rightarrow \R^*$$ is a group morphism. 
  \begin{Proposition}
  	Let $ii_{\R_+^*}$ be a ($\R_+-$valued) inconsistency map on $PC_I(\R_+^*).$ Then the map $$\begin{array}{rccl}
  	ii_{det} : & PC_I(GL_n(\R)) & \rightarrow & \R_+ \\
  	& A = (a_{i,j})_{i \in I} & \mapsto & ii_{\R_+^*}\left((|det(a_{i,j})|)_{(i,j)\in I^2}\right) 
  	\end{array},$$
  	defined as a composed map $$ii_{det} = ii_{\R_+^*} \circ |det(.)|,$$ 
  	is a non-faithful, Ad-invariant inconsistency operator on $PC_I(GL_n(\R)).$
  \end{Proposition}  
  
  \begin{proof} The only non trivial part is non-faithfulness. For this, let us give a counter-example. Let $A\in PC_3(GL_2(\R))$ defined by $a_{1,2}=a_{2,3}= -a_{1,3}=I_2.$ Then $$a_{1,2}a_{2,3} \neq a_{1,3}$$
  	where as $$|det(A)|= \left(\begin{array}{ccc} 1&1&1 \\ 1&1&1 \\1&1&1 \end{array} \right) \in CPC_3(\R_+^*),$$ thus $ii_{det}(A)=0.$
  \end{proof}
  
  But this class of inconsistency maps is not the only one of interest, even if $det$ generates $Hom(GL_n(\R),\R^*)$ in the category of groups. This toy example shows how situations expressed by a (more) complex group such as $GL_n(\R)$ cannot be reduced straightway (one would say naively) to a more simple abelian group such as $\R_+^*.$
  
  \subsection{Error in cartography and in tunnel building}	
  
  Let us now describe two situations where non-abelian groups rise naturally in a description of errors by the pairwise comparisons method. For these two examples, the group under consideration is a group of orientation preserving, isometric affine transformations $Is_n$ of a (finite dimensional) affine space $\R^n.$  
  \vskip 6pt
  \noindent
  \underline{Example: cartography in a forest}
One of the main features in cartography, or during the recovery of an exit path, in a forest (here in a flat land) is the lack of external point where to get a precise indication on the actual position of the observer. By the way, moves and direction changes can only be appreciated by self-evaluation, which is subject to numerous, non-compensative errors in the appreciation of the positions during the path. More precisely, 
\begin{enumerate}
	\item moves along a straight line can be evaluated by a translation, i.e. a vector in $\R^2$
	\item changes of direction can be evaluated as rotations, centrered at the position of the observer.
\end{enumerate}
Gathering theses two aspects, a path in teh forest can be assimilated to  succession of moves transcribed by elements $a_{1,2},... a_{n-1,n}$ of the group generated by planar rotations and translations. This group is the group of orientation preserving, isometric affine maps in $\R^2.$  Thus if one can evaluate the $n-th$ position with respect to the initial one, this gives another element $a_{1,n}.$ Due to successive errors, very often $$a_{1,n} \neq a_{1,2}a_{2,3}...a_{n-1,n}.$$
This is exactly a situation of inconsistency where coefficients are in the non-abelian group $Is_2.$ 
\vskip 12pt
\noindent
\underline{Example: error in tunnel building}
  The situation is the same in tunnel building, where the surveyors need to indicate, at each step of perforation of a tunnel, in which direction one has to correct the next perforation step underground. Each tunnel starting from each side must meet exactly at the end of the process. For the same reasons as in previous example, the (non-abelian) group under consideration here is $Is_3.$ Currently, the admissible error is in the range of 1 cm per 100 m of tunnel. This error is admissible inconsistency.   
  \subsection{Perspective in image processing}
  Let $[abcd]$ be a 3 simplex (tetrahedron). Let $\omega \in \R^3 \backslash [abcd],$ and let $P_\omega$ be a (projection) plan, such that $\omega \notin P_\omega.$ In projective perspective, the projection of $x \in \R^3 - \{\omega\},$ is $x_0 \in P_\omega$ such that 
  $$ \{x_0\} = (\omega, x) \cap P_{\omega} \quad \hbox{ (if it exists). }$$
  Let us recall that projections exist in a ``generic'' way, that is, $x_0$ exists unless
  $(x,\omega)$ and $P_\omega$ are parallel. We also assume that $\omega$ is ``far enough'' (at the ``optical infinity'') so that, in first approximation, projections can be asssimilated to affine projections on $P_\omega.$

  Let us consider for simplicity, first, a tetrahedron $[abcd]\subset \R^3,$ and let $a_0, b_0, c_0$ and $d_0$ the corresponding projections with respect to $\omega.$
  Any 2-simplex $[xyz]$ of $[abcd]$ projects to a 2-simplex $[x_0 y_0 z_0] $ in $P_\omega.$ 
  For another choice $(\omega',P_{\omega'}), $ we get other projections $[x_0' y_0' z_0'] $ of $[xyz].$
  \begin{figure}
  	\centering
  	\includegraphics[width=0.7\textwidth]{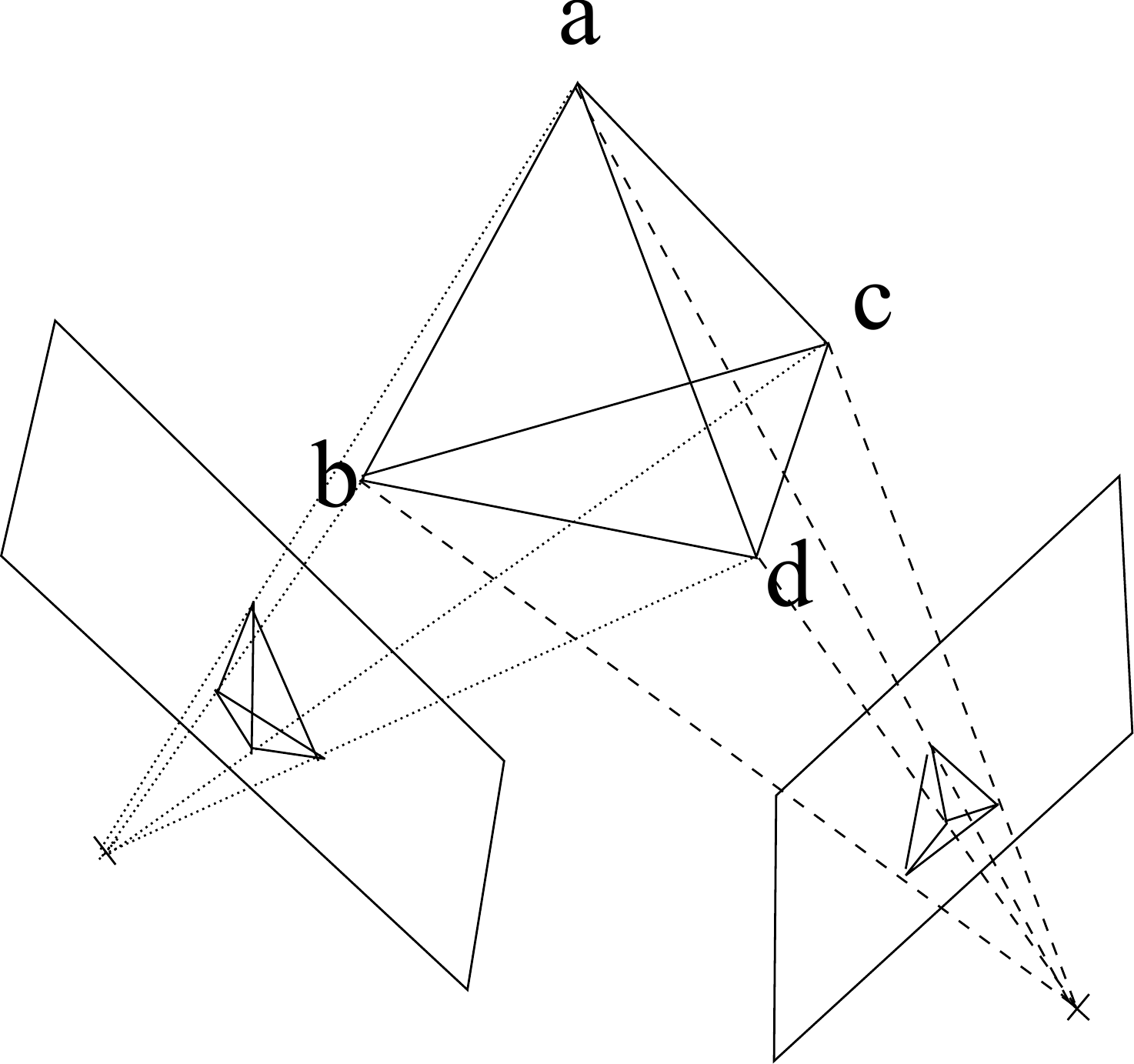}
  	\caption[abcd]{Two projections of $[abcd]$}
  	\label{fig:proj}
  \end{figure}
  Now, identifying $P_\omega$ and $P_{\omega'},$ with the standard plan $\R^2$ by an arbitrary choice of coordinates, we give the numbers 1,2,3 and 4 to the faces resp. $[abc],$ $[abd],$ $[acd]$ and $[bcd]$ and, in $\R^2,$ define $a_{i,j}$ as the unique affine  map $ \R^2 \rightarrow \R^2$ which transforms the $i-$th face to the $j-$th face for $\omega-$the projection, and  $a'_{i,j}$ the corresponding coefficient for the $\omega'-$projection. 
  Let us now consider the unique affine maps $\lambda_i$ which transforms the $i-th$ face form the $\omega-$projection to the $\omega'-$projection. We then have that 
  $$ a_{i,j} = \lambda_j^{-1} a'_{i,j} \lambda_i$$ for any index, and the PC-matrices $(a_{i,j})$ and $(a'_{i,j})$ are consistent, but the tetrahedron is flat if and only if  $$\lambda_1=\lambda_2=\lambda_3=\lambda_4=\lambda.$$
  Thus, this leads to the following definition:
  \begin{Definition}
  	Let $G$ be the group of affine bijections of $\R^2.$
  	We call \textbf{perspective matrix} of $[abcd]$ with respect to $(\omega, P_\omega)$ and $(\omega', P_{\omega'})$ the matrix $(b_{i,j})\in PC_4(G)$ defined by $$ b_{i,j} =  \lambda_ia_{i,j}\lambda_i^{-1},$$ which is defined up to Ad-action. 
  \end{Definition}
  Then we get the trivial proposition, passing through the consistency in $PC_4(G):$
  \begin{Proposition}
  	A matric $(b_{i,j}) \in PC_n(G)$ encodes a triangulated simplicial complex in $\mathbb{R}^3$ if and only if it is in $CPC_n(G).$
  \end{Proposition}
  
  Thus perspective of a 3-simplex can be encoded into a $CPC_4(G)$ matrix. Let us now have a brief look at the situation of a more complex solid that we assume triangulated, i.e. composed by a family of simplexes glued together along their faces.

\end{document}